\documentclass[a4j,10pt]{article}
\usepackage{amsmath}
\usepackage{amssymb}
\usepackage{amsthm}
\usepackage{enumerate}
\newtheorem{theorem}{Theorem}[section]
\newtheorem{Proposition}{Proposition}[section]
\newtheorem{Lemma}{Lemma}[section]

\newtheorem{Example}{Example}[section]
\newtheorem{Remark}{Remark}[section]

\begin{document}

%
%
%
%
\title{The divisibility of zeta functions of cyclotomic function fields}
\author{Daisuke Shiomi}
\date{}
\maketitle

%
%
%
%
%
%
%
%
%
%
%
%
\section{Introduction}
\quad
Let $\mathbb{F}_q$ be the finite field with $q$ elements of characteristic $p$.
Let $k=\mathbb{F}_q(T)$ be the rational function field over $\mathbb{F}_q$, and 
put $A=\mathbb{F}_q[T]$.
Let $A^+$ be the set of all monic polynomials of $A$.
For $m \in A^+$, let $K_m$
be the $m$th cyclotomic function field (see Section 2).
Let $h_m$ be the class number of $K_m$.
It is known that $h_m$ splits up into $h_m^-h_m^+$,
where $h_m^+$ is the class number of the 
maximal real subfield of $K_m$. 

In \cite{go1}, 
Goss introduced an analogue of Bernoulli numbers in function fields, and 
 gave a Kummer-type criterion for $h_m$.
For $n \ge 1$, 
we set 
%
%
%
%
\begin{eqnarray}
B_n=
\left\{
\begin{array}{ll}
\displaystyle\sum_{i=0}^{\infty}s_i(n) & \mbox{ if $n \not\equiv 0 \mod q-1 $,} \\[14pt]
\displaystyle\sum_{i=0}^{\infty} -is_i(n)&
 \mbox{ if $n \equiv 0 \mod q-1 $}, 
\end{array}
\right.
\end{eqnarray}
where
\begin{eqnarray*}
s_i(n)=\sum_{a \in A^+ \atop \deg a=i} a^n.
\end{eqnarray*}
We see that $s_i(n)=0$ for large $i$ (see Proposition \ref{prop31}).
Hence $B_n$ is a polynomial in $A$, which is called 
the Bernoulli-Goss polynomial.
%
%
%
%
\begin{theorem}{ \rm(cf. Theorem 3.4 in \cite{go1})}
\label{th11}
Assume that $m \in A^+$ is irreducible of degree $d$. Then,
\begin{itemize}
\item[(1)] $p \mid h_m^-$  if and only if
$m \mid B_n$ for some
$ 1 \le n \le q^d-2 \;( \; n \not\equiv 0 \text{\; \rm mod } q-1)$.
 \item[(2)]
 $p \mid h_m^+$  if and only if
 $m \mid B_n$ for some
$ 1 \le n \le q^d-2 \;( \; n \equiv 0 \text{\; \rm mod } q-1)$.
\end{itemize}
\end{theorem}

As an application of Theorem \ref{th11}, Goss also
showed that there exit infinitely many irreducible $m \in A^+$ with $p \mid h_m^-$
when $q=p\ge 3$. 
Feng \cite{fe} generalized the result of Goss,  and proved that
%
%
%
%
%
\begin{theorem}{\rm (cf. \cite{fe})}
\label{th12}
\begin{itemize}
\item[(1)]
If $q>2$, then
there exit infinitely many irreducible $m \in A^+$
with $p \mid h_m^-$.
\item[(2)]
There exit infinitely many irreducible $m \in A^+$
with $p \mid h_m^+$.
\end{itemize}
\end{theorem}

%
%
%
%
In this paper, we generalize Bernoulli-Goss polynomials, and 
give a criterion on the divisibility of zeta functions of cyclotomic function fields (see Theorem \ref{th41}).
As an application of our criterion, we generalize Theorem \ref{th12}
from the view point of zeta functions (see Theorem \ref{th51}). 

%
%
%
%
%
%
\section{Cyclotomic function fields}
\quad \; In this section, we review some basic facts about 
cyclotomic function fields and their zeta functions. 
For details and proofs, see \cite{ha}, \cite{go2}, \cite{ro}.

Let $\bar{k}$ be an algebraic closure of $k$. 
For $x \in \bar{k}$ and $m \in A$,  we define the following action:
%
%
\begin{eqnarray*}
m * x= m(\varphi+\mu)(x),
\end{eqnarray*}
where $\mathbb{F}_q$-linear isomorphisms $\varphi$, $\mu$ are defined by 
 \begin{eqnarray*}
 \varphi: \bar{k} \rightarrow \bar{k} \; (x \mapsto x^q), \\[5pt]
 \mu:\bar{k} \rightarrow \bar{k} \; (x\mapsto Tx).
 \end{eqnarray*}
 
By the above action, $\bar{k}$ becomes $A$-module, which is called the Carlitz module.
For $m \in A^+$, we set
 \begin{eqnarray*}
 \Lambda_m=\{ x \in \bar{k} \; | \; m*x=0 \}.
 \end{eqnarray*}
 %
 %
The field $K_m=k(\Lambda_m)$ is called  the
$m$th cyclotomic function field.
One showed that $K_m/k$ is a Galois extension, and
%
%
%
%
%
%
\begin{eqnarray}
\text{Gal}(K_m/k)\simeq (A/mA)^{\times},
\label{eq21}
\end{eqnarray}
where $\text{Gal}(K_m/k)$ is the Galois group of $K_m/k$.
Let $P_{\infty}$ be the prime of $k$ with the valuation 
$\text{ord}_{\infty}$ satisfying $\text{ord}_{\infty}(1/T)=1$.
Let 
%
%
\begin{eqnarray*}
K_m^+=k_{\infty} \cap K_m,
\end{eqnarray*}
where $k_{\infty} $ is the  completion of $k$ by $P_{\infty}$.
We call $K_m^+$ the maximal real subfield of $K_m$.
By the isomorphism (\ref{eq21}), 
we regard $\mathbb{F}_q^{\times} \subseteq \text{Gal}(K_m/k)$. 
It is known that  $K_m^+$ is the intermediate field of $K_m/k$
corresponding to $\mathbb{F}_q^{\times}$.
In particular, we see that $K_m=K_m^+$ when $q=2$. 

%
%
Next we consider the zeta functions of $K_m$ and $K_m^+$.  
For a global function field $K$ over $\mathbb{F}_q$, 
the zeta function of $K$ is defined by
%
%
%
%
\begin{eqnarray*}
\zeta(s,K)=\prod_{\mathfrak{p}:\text{\rm prime}} \Bigl(1-\frac{1}
{{N \mathfrak{p}}^s}\Bigr)^{-1},
\end{eqnarray*}
where $\mathfrak{p}$ runs through all primes of $K$, 
and $N \mathfrak{p}$ is the number of elements of 
the residue class field of  $\mathfrak{p}$.
%
%
%
%
%
%
%
%
%
\begin{theorem}(cf. Theorem 5.9 in \cite{ro})
Let $K$ be a global function field over $\mathbb{F}_q$.
Then there exists
$Z_K(u) \in \mathbb{Z}[u]$ such that
\begin{eqnarray*}
\zeta(s,K)=\frac{Z_K(q^{-s})}{(1-q^{-s})(1-q^{1-s})}.
\end{eqnarray*}
Moreover $Z_K(1)=h_K$, where $h_K$ is the class number of $K$.
\label{th21}
\end{theorem}
By  Theorem \ref{th21}, there exist
$Z_m(u), \; Z_m^{(+)}(u) \in \mathbb{Z}[u]$ such that
\begin{eqnarray*}
\zeta(s,K_m)=\frac{Z_m(q^{-s})}{(1-q^{-s})(1-q^{1-s})},  \\ \\
\zeta(s,K_m^+)=\frac{Z_m^{(+)}(q^{-s})}{(1-q^{-s})(1-q^{1-s})}.
\end{eqnarray*}
One shows that $Z_m^{(+)}(u) \mid Z_m(u)$ (cf. Section 2 in  \cite{sh1}). 
Put
\begin{eqnarray*}
Z_m^{(-)}(u)=\frac{Z_m(u)}{Z_m^{(+)}(u)}.
\end{eqnarray*}
By Theorem \ref{th21}, we have  
%
%
\begin{eqnarray}
\label{eq22}
h_m^-=Z_m^{(-)}(1), \;\;\; h_m^+=Z_m^{(+)}(1).
\end{eqnarray}

%
%
%
%
%
%
%
%
\section{Gekeler's results}
\quad 
For 
an integer
$
n=a_0+a_1q+\cdots+a_{d-1}q^{d-1} \;(0 \le a_i \le q-1)
$, 
we put 
\[
l(n)=a_0+a_1+\cdots +a_{d-1},
\]
and define $e_i$ ($1 \le i \le l(n)$) as follows:
%
%
\begin{eqnarray*}
n=\sum_{i=1}^{l(n)} q^{e_i} \;\;\;(0 \le e_i \le e_{i+1},\;\;  e_i<e_{i+q-1}).
\end{eqnarray*}
We set
\begin{eqnarray*}
\rho(n)=
\left\{
\begin{array}{ll}
-\infty  & \text{if } l(n)<q-1,\\\\
 n-\sum_{i=1}^{q-1}q^{e_i} & \text{otherwise}. 
\end{array}
\right.
\end{eqnarray*}
Moreover
$\rho(-\infty)=-\infty$, 
$\rho^{(0)}(n)=n$, and 
$\rho^{(i)}=\rho^{(i-1)} \circ \rho$.
In \cite{ge}, Gekeler gave the following result about the
degree of $s_i(n)$. 
%
%
%
%
%
%
\begin{Proposition}(cf. Proposition 2.11 in \cite{ge})
\label{prop31}
\begin{itemize}
\item[(1)]
For $i \ge 1$, we have
%
%
\begin{eqnarray}
\label{eq31}
\deg_T(s_i(n)) \le \rho^{(1)}(n) + \rho^{(2)}(n)+ \cdots +\rho^{(i)}(n).
\end{eqnarray}
In particular,  we have $s_i(n)=0$
if $l(n)/(q-1) <i$. 
\item[(2)] 
The equality (\ref{eq31}) holds if
%
%
\begin{eqnarray}
\label{eq32}
\binom{n}{\rho^{(s)}(n)} \not\equiv 0 \mod p ~~~(s=1,2,...,i),
\end{eqnarray}
where $\binom{*}{*}$ is a binomial coefficient.
\end{itemize}
\end{Proposition}
%
%
%
%
%
Let
\begin{eqnarray*}
C_n(u)=\sum_{i=0}^{\infty}s_i(n)u^i.
\end{eqnarray*}
By Proposition \ref{prop31}, we have $C_n(u) \in A[u]$, and 
\begin{eqnarray}
\label{eq33}
\deg_u C_n(u) \le \left[ \frac{l(n)}{q-1} \right],
\end{eqnarray}
where $[x]$ is the greatest integer that is less than or equal to $x$. 
For $n \ge 1$, we set
\begin{eqnarray*}
B_n(u)=
\left\{
\begin{array}{ll}
C_n(u)=\displaystyle\sum_{i=0}^{\infty}s_i(n)u^i  & \mbox{ if $n \not\equiv 0 \mod q-1 $,} \\ \\
\displaystyle\frac{C_n(u)}{1-u}=\displaystyle\sum_{i=0}^{\infty}\left(\sum_{j=0}^{i}s_j(n)\right)u^i  &
 \mbox{ if $n \equiv 0 \mod q-1 $.}
\end{array}
\right.
\end{eqnarray*}
\begin{Remark}
Supposed that $n \equiv 0 \mod q-1$.
Then $C_n(1)=0$ (cf. Lemma 6.1 in \cite{ge}). 
It follows that $B_n(u) \in A[u]$.
\end{Remark}
%
%
%
%
%
%
By the ineqality (\ref{eq33}), we have 
\begin{Proposition}
\label{prop32}
\quad

\begin{itemize}
\item[(1)] If $  n \not\equiv 0 \mod q-1$, then $\deg_u B_n(u) \le \Bigl[ \frac{l(n)}{q-1} \Bigr]$.
\item[(2)]  If  $n \equiv 0 \mod q-1$, then $\deg_u B_n(u) \le \Bigl[ \frac{l(n)}{q-1} \Bigr]-1$.
\end{itemize}
\end{Proposition}
%
%
%
%
\begin{Remark}
Supposed that  $n \equiv 0 \mod q-1$. 
By considering the derivative of  $(1-u)B_n(u)=C_n(u)$, we have
\begin{eqnarray*}
B_n(1)=\sum_{i=0}^{\infty}-is_i(n)=B_n. 
\end{eqnarray*}
Therefore $B_n(1)=B_n$ for all $n \ge 1$. 
\end{Remark}

%
%
%
%
%
%
%
%
%
\section{The divisibility of zeta functions}
\quad
Assume that $m \in A^+$  is irreducible of degree $d$.
Put
%
%
\begin{eqnarray*}
\bar{Z}_m^{(-)}(u) =Z_m^{(-)}(u) \text{ mod }p, \;\;\;
\bar{Z}_m^{(+)}(u) = Z_m^{(+)}(u) \text{ mod }p.
\end{eqnarray*}
Then 
$
\bar{Z}_m^{(-)}(u),  \; \bar{Z}_m^{(+)}(u)  \in \mathbb{F}_p[u] \subseteq (A/mA)[u]. 
$
We set
\begin{eqnarray*}
F_d^{(-)}(u)&=&\prod_{n=1\atop n \not\equiv 0 \text{ mod }q-1}^{q^d-2}B_n(u), \\ \\
F_d^{(+)}(u)&=&\prod_{n=1\atop n \equiv 0 \text{ mod }q-1}^{q^d-2}B_n(u).
\end{eqnarray*}
For $G(u) \in A[u]$, we write
$\bar{G}(u)=G(u) \mod m \in (A/mA)[u]$.
Then we have
%
%
%
%
\begin{Proposition}(cf. Proposition 2.2 in \cite{sh2})
\label{prop41}
\begin{itemize}
\item[(1)] $\bar{F}_d^{(-)}(u) = \bar{Z}_m^{(-)}(u)$, 
\item[(2)] $\bar{F}_d^{(+)}(u) = \bar{Z}_m^{(+)}(u)$.
\end{itemize}
In particular, we have
$
\bar{F}_d^{(-)}(u), \;\bar{F}_d^{(+)}(u) \in \mathbb{F}_p[u].
$
\end{Proposition}

Let $f(u) \in \mathbb{F}_p[T]$ be a monic irreducible polynomial.
Let $\alpha$ be a root of $f(u)$, and $s=[\mathbb{F}_q(\alpha):\mathbb{F}_q]$. 
For $n \ge 1$, we define
\begin{eqnarray*}
B_{n}^{\alpha}=N_{\mathbb{F}_{q^s}(T)/\mathbb{F}_q(T)}(B_n(\alpha)) \in A,
\end{eqnarray*}
where $N_{\mathbb{F}_{q^s}(T)/\mathbb{F}_q(T)}$ is the norm 
from $\mathbb{F}_{q^s}(T)$ to $\mathbb{F}_{q}(T)$. 
Let $\alpha_1, \; \alpha_2, ...., \alpha_s$ be all conjugates
of $\alpha$ over $\mathbb{F}_q$. Then we have
\begin{eqnarray*}
B_{n}^{\alpha}=\prod_{i=1}^{s}B_n(\alpha_i).
\end{eqnarray*}
Now we state the main result of this paper. 
%
%
%
%
\begin{theorem}
\label{th41}
\quad 
\begin{itemize}
\item[(1)] $f(u) \mid \bar{Z}_m^{(-)}(u)$  if and only if
$m \mid B^{\alpha}_n$ for some
$ 1 \le n \le q^d-2$ with $n \not\equiv 0 \text{ \rm mod } q-1$.
\item[(2)]
$f(u) \mid \bar{Z}_m^{(+)}(u)$  if and only if
$m \mid B^{\alpha}_{n}$ for some
$ 1 \le n \le q^d-2$ with $n \equiv 0 \text{ \rm mod } q-1$.
\end{itemize}
\end{theorem}
%
%
%
%
%
%
\begin{proof}
Let $A_s=\mathbb{F}_{q^s}[T]$. 
Let $\mathfrak{m} \in A_s$ be a monic irreducible polynomial with $\mathfrak{m} \mid m$.
We regard
$\mathbb{F}_{q^s}$ and  $A/mA$ as subfields of $ A_s/\mathfrak{m}A_s$.
Then
\begin{eqnarray*}
m \mid B^{\alpha}_{n} \iff
\bar{B}_n(\alpha_i) =0 \;\; \text{ in }
 A_s/\mathfrak{m}A_s \; (\text{for some $1 \le i \le s$}).
\end{eqnarray*}
Hence the following conditions are equivalent.
%
%
%
%
\begin{itemize}
\item[(i)] $m\mid B^{\alpha}_{n}$ for some  $1 \le n \le q^d-2$ 
with $n \not\equiv 0 \mod q-1$.
\item[(ii)]$\bar{F}_d^{(-)}(\alpha_i) =0$  in $A_s/\mathfrak{m}A_s$
for some $1 \le i \le s$.
\end{itemize}
Noting that $f(u)$ is the minimal polynomial of $\alpha_i$ over $\mathbb{F}_p$, 
the condition (ii) is equivalent to $f(u) \mid \bar{F}_d^{(-)}(u)$
(note that $\bar{F}_d^{(-)}(u) \in \mathbb{F}_p[u]$). 
Hence the claim (1) follows from Proposition \ref{prop41}. 

By the same argument, we have the claim (2). 
\end{proof}

%
%
\begin{Remark}
Supposed that $f(u)=u-1$. 
By the equality (\ref{eq22}), we have 
\begin{eqnarray}
p\; | \; h_m^- &\iff& f(u) \mid \bar{Z}_m^{(-)}(u),  \label{eq41}\\[5pt]
p\; | \; h_m^+ &\iff& f(u) \mid \bar{Z}_m^{(+)}(u). \label{eq42}
\end{eqnarray}
We see that $B^{\alpha}_{n}=B_n$ when $\alpha=1$. 
Thus Theorem \ref{th41} yields Theorem \ref{th11}. 
\end{Remark}
%
%
%
%
\begin{Example}
Let $q=2$ and $f(u)=u^2+u+1=(u-\alpha)(u-\beta)$.
Supposed that $1 \le n \le 2^3-2=6$. 
 By Proposition \ref{prop32},  we have 
$
B_n(u) =1+(1+s_1(n))u.
$
Therefore, 
\begin{eqnarray*}
B^{\alpha}_{n}=B_n(\alpha)B_n(\beta)=1+s_1(n)+s_1(n)^2.
\end{eqnarray*}
The factorizations of $B^{\alpha}_{n}$ in $A$ are as follows:
\begin{eqnarray*}
B^{\alpha}_{1}&=&B^{\alpha}_{2}=B^{\alpha}_{4}=1, \\[5pt]
B^{\alpha}_{3}&=&T^4+T+1, \\[5pt] 
B^{\alpha}_{5}&=&(T^4+T^3+1)(T^4+T^3+T^2+T+1), \\[5pt]
B^{\alpha}_{6}&=&(T^4+T+1)^2. 
\end{eqnarray*}
Therefore, by Theorem \ref{th41}, we have $f(u) \nmid \bar{Z}_m^{(+)}(u)$ for all irreducible
$m \in A^+$
of $\deg m \le 3$.
\end{Example}

%
%
%
%
\begin{Example}
Let $q=3$ and $f(u)=u^2+1=(u-\alpha)(u-\beta)$.
Supposed that $1 \le n \le 3^2-2=7 \;(n \not\equiv 0 \text{ \rm mod }2 )$.
Then we have 
$
B_n(u)=1+s_1(n)u.
$
Therefore, 
\begin{eqnarray*}
B^{\alpha}_{n}=(1+s_1(n)\alpha)(1+s_1(n)\beta)=1+s_1(n)^2. 
\end{eqnarray*}
The factorizations of $B_n^{\alpha}$ in $A$ are as follows:
\begin{eqnarray*}
B^{\alpha}_{1}&=& B^{\alpha}_{3}=1, \\[5pt]
B^{\alpha}_{5}&=&B^{\alpha}_{7}=(T^2+1)(T^2+T+2)(T^2+2T+2).
\end{eqnarray*}
Therefore  $f(u) \mid \bar{Z}^{(-)}_m(u)$ for 
$
m=T^2+1, \; T^2+T+2, \; T^2+2T+2.
$
\end{Example}
%
%
%
%
%
%
%
%
\section{A generalization of Theorem \ref{th12}}
\quad
Let $f(u) \in \mathbb{F}_p[T]$ be a monic irreducible polynomial with
a root $\alpha$. 
%
%
\begin{Lemma}
\label{lem51}
Assume that  $m \in A^+$ is irreducible of degree $d$.
Supposed that $n_1 \equiv n_2 \mod q^d-1$.  Then 
$B^{\alpha}_{n_1} \equiv B^{\alpha}_{n_2} \mod m$.
\end{Lemma}
\begin{proof}
Since $A/mA$ is a finite field with $q^d$ elements,
we have
\begin{eqnarray*}
a^{n_1} \equiv a^{n_2} \mod m \;\;\; \;\;\text{for all } a \in A.
\end{eqnarray*}
Hence
$s_i(n_1) \equiv s_i(n_2) \mod m ~(i=0,1,2,...)$.
This implies
$
B^{\alpha}_{n_1} \equiv B^{\alpha}_{n_2} \mod m.
$
\end{proof}

%
%
%
%
\begin{Lemma}
\label{lem52}
$B^{\alpha}_{q}=B^{\alpha}_{2(q-1)}=1$. 
\end{Lemma}
\begin{proof}
By Proposition \ref{prop32},  we have
$
B_q(u)=B_{2(q-1)}(u)=1.
$
It follows that $B^{\alpha}_{q}=B^{\alpha}_{2(q-1)}=1$. 
\end{proof}
%
%
%
%
\begin{Lemma}
\label{lem53}
Let $d$ be a positive integer.
\begin{itemize}
\item[(1)] If $n=1+(q-1)q^d$, then $\deg_T B^{\alpha}_{n}>0$.
\item[(2)]  If $n=(q-1)+(q-1)q^d$, then $\deg_T B^{\alpha}_{n}>0$.
\end{itemize}
\end{Lemma}
%
%
\begin{proof}
Let $n=1+(q-1)q^d$.  We see that 
\begin{eqnarray*}
\binom{n}{\rho^{(1)}(n)} =\binom{1+(q-1)q^d}{q^d}  \equiv -1 \mod p.
\end{eqnarray*}
By Proposition \ref{prop31}, 
we have $\deg_T s_1(n)=\rho^{(1)}(n)=q^d>0$.
By Proposition \ref{prop32}, we have
$
B_n(u)=1+s_1(n)u.
$
It follows that
\begin{eqnarray*}
\deg_T B^{\alpha}_{n} \ge \deg_T B_n(\alpha) = \deg_T s_1(n)
\end{eqnarray*}
in $\mathbb{F}_{q^s}[T]$.
Hence we have the  claim (1).
%
%
Similarly, 
we have the claim (2).
\end{proof}

%
%
%
%
%
\begin{theorem}
\label{th51}
\quad
\begin{itemize}
\item[(1)]
Assume that  $q>2$. Then there exit infinitely many  irreducible  $m \in A^+$
with 
$f(u) \mid \bar{Z}_m^{(-)}(u)$. 
\item[(2)]
There exit infinitely many  irreducible  $m \in A^+$
with $f(u) \mid \bar{Z}_m^{(+)}(u)$. 
\end{itemize}
\end{theorem}

\begin{proof}
It is sufficient to show that there exists
 an irreducible $m \in A^+$ with $\deg m > d$ and  
$f(u) \mid \bar{Z}_m^{(-)}(u)$ (resp. $f(u) \mid \bar{Z}_m^{(+)}(u)$)
for all $d \ge 2$. 

%
%
(1) First we consider the minus part. 
Put $n=1+(q-1)q^{d!}$. By Lemma \ref{lem53}, 
 we have $\deg B^{\alpha}_{n}  >0$.
Take an irreducible  $m\in A^+$ with $m \mid B^{\alpha}_{n}$.
Supposed that $d_1:=\deg m \le d$. Then  $n \equiv q \mod q^{d_1}-1$.
By Lemma \ref{lem51} and \ref{lem52}, we have 
\begin{eqnarray*}
B^{\alpha}_{n}  \equiv B^{\alpha}_{q}  \equiv 1 \mod m.
\end{eqnarray*}
This contradicts $m \mid B^{\alpha}_{n} $. So $d_1>d$.
Take $a, b \in \mathbb{Z}$ with
\begin{eqnarray*}
n=a(q^{d_1}-1)+b, \text{~~} 0\le b\le q^{d_1}-2.
\end{eqnarray*}
Then $b \not\equiv 0 \mod q-1$ and 
$
B^{\alpha}_{b}\equiv B^{\alpha}_{n} \equiv 0 \mod m.
$
By Theorem\ref{th41}, we obtain $f(u) \mid \bar{Z}_m^{(-)}(u)$. 
\quad \\ \quad
%
%
(2) Next we consider the plus part. 
Let $n=(q-1)+(q-1)q^{d!}$.
By the same argument of (1),
we can take an irreducible $m \in A^+$ such that 
\begin{eqnarray*}
d_2:=\deg m >d \text{\; and \;}m \mid B^{\alpha}_{n}.
\end{eqnarray*}
Let $a, b \in \mathbb{Z}$ with
\begin{eqnarray*}
n=a(q^{d_2}-1)+b, \;\; 0\le b\le q^{d_2}-2.
\end{eqnarray*}
Then we have
\begin{eqnarray*}
b \equiv 0 \mod q-1,\;\;  B^{\alpha}_{b} \equiv 0 \mod m.
\end{eqnarray*}
Supposed that $b=0$.  Then there exists $c \in \mathbb{Z} \; ( \; 0\le c<d_2 \;)$
such that 
\begin{eqnarray*}
 (q-1)+(q-1)q^{c} \equiv 0 \mod q^{d_2}-1.
\end{eqnarray*} 
This contradicts $d_2 >d=2$. So $b>0$. 
It follows that $f(u) \mid \bar{Z}_m^{(+)}(u)$.
\end{proof}
%
%
%
%
%
\begin{Remark}
If $q=2$,
then $Z_m^{(-)}(u)=1$
because $K_m=K_m^+$.
\end{Remark}
\begin{Remark}
We consider the case $f(u)=u-1$. By the equivalences (\ref{eq41}) and (\ref{eq42}), 
we see that Theorem \ref{th51}
leads Theorem \ref{th12}.
\end{Remark}

\text{ }\\
Daisuke Shiomi \\
Department of Mathematical Sciences \\
Faculty of Science, Yamagata University \\
Kojirakawa-machi 1-4-12 \\
Yamagata 990-8560, Japan \\
shiomi@sci.kj.yamagata-u.ac.jp

\end{document}